\documentclass[12pt]{article}
\usepackage[centertags]{amsmath}
\usepackage{amsfonts}
\usepackage{amssymb}
\usepackage{amsthm}
\usepackage{newlfont}
\usepackage{graphicx}
\usepackage{t1enc}
\usepackage[english]{babel}
\usepackage{bm}
\usepackage{cite}
\newlength{\defbaselineskip}
\newcommand{\setlinespacing}[1]%
           {\setlength{\baselineskip}{#1 \defbaselineskip}}

\newcommand{\dd }{\displaystyle}

\newcommand{\s }{\\[0.2cm]}

\newcommand{\h }{\hspace*{18pt}}

\newcommand{\q}{\quad}

\newcommand{\bq }{\begin{equation}}
\newcommand{\eq }{\end{equation}}
\newcommand{\bbb }{\begin{eqnarray}}
\newcommand{\eee }{\end{eqnarray}}
\newcommand{\bb }{\begin{eqnarray*}}
\newcommand{\ee }{\end{eqnarray*}}
\newcommand{\ed }{\end{document}}
\theoremstyle{plain}
\newtheorem{thm}{Theorem}

\newtheorem{cor} {Corollary}
\newtheorem{rem}{Remark}
\theoremstyle{definition}

\theoremstyle{example}

\begin{document}
\begin{center}
{SOME NEW IDENTITIES OF GENERALIZED FIBONACCI AND GENERALIZED PELL
NUMBERS VIA A NEW TYPE OF NUMBERS}\s
 {\textbf{\textsf{ W.M. Abd-Elhameed$^{1,2}$and N.A. Zeyada$^{1,2}$}}}\\[0.2cm]
$^1$Department of Mathematics, Faculty of Science, University
of Jeddah, Jeddah, Saudi Arabia\\[0.2cm]
\centerline{\begin{small} $^2$ Department of Mathematics, Faculty of
Science, Cairo University, Giza, Egypt
\end{small}}
 \centerline{\begin{small}
  \q E-mail: walee$_{-}$9@yahoo.com (Abd-Elhameed),
 E-mail: nasmz@yahoo.com (Zeyada),\end{small}}
\end{center}
\hrule
\[\]
\begin{abstract}
 This paper is concerned with developing some new identities of
 generalized Fibonacci numbers and generalized Pell numbers. A new class of generalized numbers
 is introduced for this purpose. The two well-known
 identities of Sury and Marques which are recently developed are deduced as special cases.
 Moreover, some other
 interesting identities involving the celebrated Fibonacci,
 Lucas, Pell and Pell-Lucas numbers are also deduced.
 \end{abstract}

 \s

\section{Introduction}
\h Many number sequences can be generated by recurrence relations of
order two, among these numbers are the celebrated Fibonacci, Lucas,
Pell and Pell-Lucas numbers. It is well-known that these sequences
of numbers have important parts in mathematics. They are of
fundamental importance in the fields of combinatorics and number
theory, see for example
\cite{keskin2011fibonacci,koshy2011fibonacci}. Recently, the studies
concerned with various generalizations of Fibonacci and Lucas
numbers have attracted a number of authors, see for example
\cite{wang2015some,gulec2013new,
falcon2007fibonacci,yazlik2012note,melham1999sums}.\s
   \h The $n$-th Fibonacci and Lucas numbers can be generated respectively, with the aid
   of the two recurrence relations:
   \[F_{n+2}=F_{n+1}+F_{n},\quad F_{0}=0,\, F_{1}=1,\]
   and
     \[L_{n+2}=L_{n+1}+L_{n},\quad L_{0}=2,\, L_{1}=1.\]
     A few terms of Fibonacci sequence are\\
     \[0,1,1,2,3,5,8,13,21,\ldots,\]
     while a few terms of Lucas sequence are
     \[2,1,3,4,7,11,18,29,47,\ldots.\]
     Note that the Lucas numbers $L_{n}$ are linked with the Fibonacci numbers
   $F_{n}$ by the relation
   \[L_{n}=F_{n-1}+F_{n+1},n\ge 1.\]

     The generation of the two sequences of Fibonacci and Lucas numbers can be unified via the
     sequence of generalized Fibonacci numbers $(G^{a,b}_{n})_{n\ge 0}$.
     These numbers can be generated with the aid of the following recurrence relation:
      \bq
      \label{GenFib}
      G^{a,b}_{n+2}=G^{a,b}_{n+1}+G^{a,b}_{n},\quad G^{a,b}_{0}=b-a,G^{a,b}_{1}=a.
       \eq
      We note that the Fibonacci
      and Lucas sequences are special classes of $(G^{a,b}_{n})_{n\ge 0}$.
      Explicitly, we have
      \[F_{n}=G^{1,1}_{n},\quad L_{n}=G^{1,3}_{n}.\]
   An important identity of $G^{a,b}_{n}$ is
   \[
   G^{a,b}_{n+2}=a\, F_{n}+b\, F_{n+1}.
   \]
   For properties of Fibonacci, Lucas and generalized Fibonacci numbers, one
   can be referred to the beautiful book of Koshy
   \cite{koshy2011fibonacci}.\\
  There are other two important sequences of numbers, namely Pell numbers
  $(P_{n})_{n\ge 0}$, and Pell-Lucas numbers $(Q_{n})_{n\ge 0}$.
  These two sequences may be generated by the following two recurrence
  relations:
   \[P_{n+2}=2\, P_{n+1}+P_{n},\quad P_{0}=0,\, P_{1}=1,\]
   and
     \[Q_{n+2}=2\, Q_{n+1}+Q_{n},\quad Q_{0}=2,\, Q_{1}=2.\]
     A few terms of Pell sequence are
     \[0,1,2,5,12,29,70,169,408,\ldots,\]
     while a few terms of Pell-Lucas sequence are
     \[2,2,6,14,34,82,198,478,1154,\ldots.\]
     Note that the numbers $Q_{n}$ are linked with the numbers $P_{n}$ by the
   relation:
   \[Q_{n}=P_{n-1}+P_{n+1},n\ge 1.\]
 Now, we introduce the sequence of generalized Pell numbers, which we
denote
     $(P^{a,b}_{n})$. This sequence may be generated by the recurrence relation
       \bq
       \label{GenPell}
       P^{a,b}_{n+2}=2\, P^{a,b}_{n+1}+P^{a,b}_{n},\quad P^{a,b}_{0}=b-2a,P^{a,b}_{1}=a.
       \eq
      It is clear that the generalized Pell sequence $(P^{a,b}_{n})_{n\ge 0}$ generalizes
      the two sequences of Pell and Pell Lucas numbers. In fact, we
      have
      \[P_{n}=P^{1,2}_{n},\quad Q_{n}=P^{2,6}_{n}.\]
   In this letter, we construct a more general sequence of numbers aiming to
   unify the construction of the two sequences of the generalized
   Fibonacci numbers $(G^{a,b}_{n})_{n\ge 0}$ in \eqref{GenFib} and the generalized Pell numbers
   $(P^{a,b}_{n})_{n\ge 0}$ in \eqref{GenPell}. For this purpose, we consider the
   generalized sequence of numbers $(U^{a,b,r}_{n})$
   generated by the recurrence relation:
   \bq
   \label{recgen}
 U^{a,b,r}_{n+2}=r\, U^{a,b,r}_{n+1}+U^{a,b,r}_{n},
 \quad U^{a,b,r}_{0}=b-r\, a,\ U^{a,b,r}_{1}=a.
 \eq
 It is clear from \eqref{recgen} that the two sequences $(G^{a,b}_{n})$ and
 $(P^{a,b}_{n})$ are particular sequences of the more general sequence
  $(U^{a,b,r}_{n})$.
 In fact, we have
 \[G^{a,b}_{n}=U^{a,b,1}_{n},\qquad P^{a,b}_{n}=U^{a,b,2}_{n}.\]
 Thus, the main advantage of introducing the sequence
 $(U^{a,b,r}_{n})$ is that the four sequences of Fibonacci $(F_{n})$, Lucas
 $(L_{n}),$ Pell
 $(P_{n}),$ Pell-Lucas
 $(Q_{n})$ numbers can be deduced as particular cases of the generalized sequence of numbers
 $(U^{a,b,r}_{n})$.

   \section{New identities of some generalized numbers}
    This section is devoted to presenting some new identities of the three generalized
    sequences of numbers $(G^{a,b}_{n}),(P^{a,b}_{n})$ and $(U^{a,b,r}_{n})$ which
    introduced
    in Section 1. We show that one of the presented identities generalize the two identities of
    Sury and Marques which are recently developed. In addition, some other
    new identities involving
    the celebrated Fibonacci, Lucas, Pell and Pell-Lucas numbers
    are also deduced.
     \begin{thm}
     \label{Thm1}
     For every nonnegative integer $m$ and for all $c\in \mathbb{R}-\{0\}$, the
     following identity is valid
     \bq
     \label{mainidentity}
     c^{m+1}\, U^{a,b,r}_{m+1}=b-r\, a+\dd\sum_{i=0}^{m}
     c^i\, \left\{(r-1)\, \, U^{a,b,r}_{i}+
     (c-1)\, U^{a,b,r}_{i+1}+U^{a,b,r}_{i-1}\right\}.
     \eq
    \end{thm}
   \begin{proof}
 We will proceed by induction. It is clear that each of the
 two sides of \eqref{mainidentity} is equal to: $(a\, c)$ in case of $m=0$. Now, assume
 that \eqref{mainidentity} is valid, hence to complete the proof, we have
 to show the validity of the following identity:
 \bq
     \label{mainidentity21}
     c^{m+2}\, U^{a,b,r}_{m+2}=b-r\, a+\dd\sum_{i=0}^{m+1}
     c^i\, \left\{(r-1)\, \, U^{a,b,r}_{i}+
     (c-1)\, U^{a,b,r}_{i+1}+U^{a,b,r}_{i-1}\right\}.
     \eq
Now, we write the right hand side of \eqref{mainidentity21} in the
form \bq
\begin{split}
 &b-r\, a+\dd\sum_{i=0}^{m+1}
     c^i\, \left\{(r-1)\, \, U^{a,b,r}_{i}+
     (c-1)\, U^{a,b,r}_{i+1}+U^{a,b,r}_{i-1}\right\}=\\
&b-r\, a+\dd\sum_{i=0}^{m}
     c^i\, \left\{(r-1)\, \, U^{a,b,r}_{i}+
     (c-1)\,
     U^{a,b,r}_{i+1}+U^{a,b,r}_{i-1}\right\}\\
     &+c^{m+1}\left[(r-1)\,
     U^{a,b,r}_{m+1}+(c-1)\, U^{a,b,r}_{m+2}+U^{a,b,r}_{m}\right].
\end{split}
\eq
 If we make use of the valid relation \eqref{mainidentity}, then
 the latter equation can be turned into
 \bq
 \begin{split}
 \label{ind2}
 &b-r\, a+\dd\sum_{i=0}^{m+1}
     c^i\, \left\{(r-1)\, \, U^{a,b,r}_{i}+
     (c-1)\, U^{a,b,r}_{i+1}+U^{a,b,r}_{i-1}\right\}=\\
&c^{m+1}\, U^{a,b,r}_{m+1}+c^{m+1}\left[(r-1)\,
     U^{a,b,r}_{m+1}+(c-1)\, U^{a,b,r}_{m+2}+U^{a,b,r}_{m}\right].
\end{split}
\eq
 Based on the recurrence relation \eqref{recgen}, it is easy to see that
\eqref{ind2} can be written alternatively as \bq
     \label{mainidentity23}
     b-r\, a+\dd\sum_{i=0}^{m+1}
     c^i\, \left\{(r-1)\, \, U^{a,b,r}_{i}+
     (c-1)\, U^{a,b,r}_{i+1}+U^{a,b,r}_{i-1}\right\}=c^{m+2}\, U^{a,b,r}_{m+2}.
     \eq
This completes the proof of Theorem \ref{Thm1}.
 \end{proof}
  \noindent
 The following identity of the generalized Fibonacci numbers $ G^{a,b}_{n}$ is
 a direct consequence of identity \eqref{mainidentity}.

    \begin{thm}
     For every nonnegative integer $m$ and for all $c\in \mathbb{R}-\{0\}$, the
     following identity holds for generalized Fibonacci numbers
     \bq
     \label{mainidentityGenFib}
     c^{m+1}\, G^{a,b}_{m+1}=b-a+\dd\sum_{i=0}^{m}
     c^i\, \left\{(c-1)\, G^{a,b}_{i+1}+G^{a,b}_{i-1}\right\}.
     \eq
    \end{thm}
    \begin{proof}
    Identity \eqref{mainidentityGenFib} follows immediately from
    identity \eqref{mainidentity} by setting $r=1$.
    \end{proof}

     Several important identities involving Fibonacci and Lucas numbers
     can be deduced as special cases of
     \eqref{mainidentityGenFib}. These identities are stated in the following
     corollaries.
    \begin{cor}
     If we set $a=b=1$ in \eqref{mainidentityGenFib}, then the following
     identity is obtained:
    \bq
     \label{id2}
     c^{m+1}\, F_{m+1}=\dd\sum_{i=0}^{m}
     c^i\, \left\{(c-1)\, F_{i+1}+F_{i-1}\right\},
     \eq
     and in particular, we have the following two important
     identities:
     \bq
     \label{Sury}
     2^{m+1}\, F_{m+1}=\dd\sum_{i=0}^{m}
     2^i\, L_{i},
     \eq
      and
    \bq
    \label{Marq}
     3^{m+1}\, F_{m+1}=\dd\sum_{i=0}^{m}
     3^i\, L_{i}+\dd\sum_{i=0}^{m+1}3^{i-1}\, F_{i}.
     \eq
     \begin{rem}
      The identity in \eqref{Sury} is in agreement with that obtained in
      \cite{sury2014polynomial,kwong2014alternate},
      while the identity \eqref{Marq} coincides with that obtained in
      \cite{marques2015new}.
     \end{rem}
 \end{cor}
     \begin{cor}
        If we set $a=1,b=3$ in \eqref{mainidentityGenFib}, then the following
     identity is obtained:
     \bq
     \label{identity3}
     c^{m+1}\, L_{m+1}=2+\dd\sum_{i=0}^{m}
     c^i\, \left\{(c-1)\, L_{i+1}+L_{i-1}\right\}.
     \eq
     \end{cor}
     \begin{cor}
    If we set $c=1$ in \eqref{mainidentityGenFib}, then the following
     identity holds for generalized Fibonacci numbers:
 \bq
     \label{id1}
     G^{a,b}_{m+1}=b-a+\dd\sum_{i=0}^{m}
      G^{a,b}_{i-1}.
     \eq
    \end{cor}
  \noindent
     Now, we state an identify involving the generalized Pell numbers $P^{a,b}_{n}$.
    \begin{thm}
     For every nonnegative integer $m$ and for all $c\in \mathbb{R}-\{0\}$, the
     following identity holds for generalized Pell numbers
     \bq
     \label{mainidentity25}
      c^{m+1}\, P^{a,b}_{m+1}=b-2\, a+\dd\sum_{i=0}^{m}
     c^i\, \left\{P^{a,b}_{i}+
     (c-1)\, P^{a,b}_{i+1}+P^{a,b}_{i-1}\right\}.
     \eq
    \end{thm}
    \begin{proof}
   Identity \eqref{mainidentity25} follows immediately from
    identity \eqref{mainidentity} by setting $r=2$.
    \end{proof}
    The following identities follow from identity
    \eqref{mainidentity25} for particular choices of its parameters.
    \begin{cor}
        If we set $a=1,b=2$ in \eqref{mainidentity25}, then the following
     identity is obtained:
     \bq
     \label{identity3}
     c^{m+1}\, P_{m+1}=\dd\sum_{i=0}^{m}
     c^i\, \left\{P_{i}+
     (c-2)\, P_{i+1}+Q_{i}\right\},
     \eq
     and in particular, we have
   \bq
     \label{identity3}
     2^{m+1}\, P_{m+1}=\dd\sum_{i=0}^{m}
     2^i\, \left\{P_{i}+Q_{i}\right\}.
     \eq
     \end{cor}
     \begin{cor}
        If we set $a=2,b=6$ in \eqref{mainidentity25}, then the following
     identity is obtained:
     \bq
     \label{identity3}
     c^{m+1}\, Q_{m+1}=2+\dd\sum_{i=0}^{m}
     c^i\, \left\{Q_{i}+
     (c-1)\, Q_{i+1}+Q_{i-1}\right\}.
     \eq
         \end{cor}
\begin{cor}
    If we set $c=1$ in \eqref{mainidentity25}, then the following
     identity holds for generalized Pell numbers
 \bq
     \label{id1}
     P^{a,b}_{m+1}=b-2\, a+\dd\sum_{i=0}^{m}
      \left\{P^{a,b}_{i}+P^{a,b}_{i-1}\right\}.
     \eq
    \end{cor}

\end{document}